\documentclass[10pt]{amsart}

\usepackage[utf8]{inputenc}
\usepackage{verbatim} 
\usepackage{paralist} 
\usepackage[scaled]{helvet} 
\usepackage{xspace}
\usepackage{amssymb,latexsym}

\theoremstyle{plain} 
\newtheorem{theorem}{Theorem}[section] 
\newtheorem{thm}{Theorem}[section] 
\newtheorem{proposition}[theorem]{Proposition}
\newtheorem{lemma}[theorem]{Lemma}
\newtheorem{coro}[theorem]{Corollary} 
 
\theoremstyle{definition} 
\newtheorem{definition}[theorem]{Definition} 
 
\theoremstyle{remark} 
\newtheorem{remark}[thm]{Remark} 
\numberwithin{equation}{section}
\numberwithin{thm}{section}

\newcommand{\na}{\mathbb{N}}
\newcommand{\en}{\mathbb{Z}}

\newcommand{\al}{\alpha}
\newcommand{\alv}{\alpha_v}

\newcommand{\sz}{\sigma(D,\al)}

\newcommand{\sgf}{semi-Grundy function\xspace}
\newcommand{\gf}{Grundy function\xspace}
\newcommand{\gfs}{Grundy functions\xspace}
\newcommand{\sk}{semi-kernel\xspace}
\usepackage{amsmath,amsthm,amssymb,amsfonts}
\usepackage[a4paper]{geometry}
\usepackage{a4wide}
\usepackage{parskip}
\usepackage{graphicx}
\usepackage{booktabs}
\usepackage{rotating}

\usepackage{subfigure}
\usepackage{hyperref}
\usepackage{color}
\usepackage{xcolor}
\usepackage{fancyhdr}
\usepackage[margin=10pt,font=small,labelfont=bf,labelsep=endash]{caption}

\date{}

\begin{document}

\begin{center}
  \title{Semi-Grundy Function, an hereditary approach to Grundy function}
  \thanks{This research was supported by grants PAPIIT-IN104717 and CONACyT
    CB-2013/219840. The second author was supported by postdoctoral position under
    the grant CONACyT 2018-000022-01EXTV at Charles University.}

  \author{
    Hortensia Galeana-S\'{a}nchez$^{1}$, Ra\'{u}l Gonz\'{a}lez-Silva$^{2}$ \\ \\
    \tiny{$^1$Instituto de Matem\'{a}ticas, UNAM, Ciudad Universitaria, M\'exico CDMX, MEXICO}\\
    \tiny{$^2$ Department of Applied Mathematics, Faculty of Mathematics and Physics, Charles University, Malostranské nám\v{e}stí 25,11800 Prague, Czech Republic.}\\
      \tiny{\texttt{\MakeLowercase{$^{1}$hgaleana@matem.unam.mx,
            $^{2}$rulo65@ciencias.unam.mx}}} }
  
\end{center}

\begin{abstract}
  Grundy functions have found many applications in a wide variety of games, in
  solving relevant problems in Game Theory. Many authors have been working on this
  topic for over many years, by example: C. Berge, P. Erdös, M. P. Schützenberger,
  R. P. Sprague.  Since the existence of a Grundy function on a digraph implies that
  it must have a kernel, the problem of deciding if a digraph has a Grundy function
  is NP-complete, and how to calculate one is not clearly answered.

  In this paper, we introduce the concept: Semi-Grundy function, which arises
  naturally from the connection between kernel and semi-kernel and the connection
  between kernel and Grundy function. We explore the relationship of this concept
  with the Grundy function, proving that for digraphs with a defining hereditary
  property is sufficient to get a semi-grundy function to obtain a Grundy
  function. Then we prove sufficient and necessary conditions for some products of
  digraphs to have a semi-Grundy function. Also, it is shown a relationship between
  the size of the semi-Grundy function obtained for the Cartesian Product and the
  size of the semi-Grundy functions of the factors. This size is an upper bound of
  the chromatic number.

  We present a family of digraphs with the following property: for each natural
  number $n\geq 2$, there is a digraph $R_n$ that has two Grundy functions such that
  the difference between their maximum values is equal to $n$. Then it is important
  to have bounds for the Grundy or semi-Grundy functions.

\end{abstract}
\maketitle
\section{Introduction.}

The concept of kernel was introduced by Von Neumann and Morgenstern \cite{vNoM44} in
the context of Game Theory. The problem of the existence of a kernel in a given
digraph has been studied by several authors, see for example
(\cite{pD80,pD87,pDhM81,mR53}). An important concept to study kernels of a digraph,
is the concept of semi-kernel, introduced in \cite{hGvN84}. The following result is
an important example of this relation:
\begin{theorem}\cite{vN71} \label{teo1}
  If every induced digraph of $D$ has a semi-kernel then $D$ has a kernel.
\end{theorem}

Another concept closely related to kernels of a digraph is Grundy functions,
introduced by P. M. Grundy in \cite{pG39}. This concept have found many applications
in Game Theory see for example (\cite{cB85,aFmL91,aF96,hL02}) and in graph theory:
\cite{pEsH03}. The relation of these concepts is shown in the following results:

\begin{theorem}\cite{cB85} \label{teo2}
  If $D$ has a Grundy function $g$, then the set $N=\{x\in V(D)|g(x)=0\}$ is a kernel
  of $D$.
\end{theorem}

\begin{theorem}\cite{cB85} \label{teo3}
  If $D$ is a kernel-perfect digraph, then $D$ possesses a Grundy function.
\end{theorem}

The problem of deciding if a digraph has kernel is NP-complete. The research of
sufficient conditions for a digraph to have kernel has been addressed by many authors
along many years, for a comprehensive survey see \cite{eBvG}. For digraphs with a
defining hereditary property is sufficient to get:
\begin{itemize}
\item a semi-kernel to obtain a kernel,
\item a kernel to obtain a Grundy function,
\end{itemize}

We introduce a new concept, namely \sgf. This is a non-negative integer function
defined on the set of vertices of a digraph. This concept generalizes that of
Grundy function and allow us to obtain Grundy functions in an easier way, in those
digraphs defined by an hereditary property.

In \cite{hGrG10} we prove sufficient and necessary conditions for the Cartesian
Product $\sz$ to have a Grundy function in terms of the existence of Grundy
function or kernel in $D$ and in each $\al_v$. Also, it is shown a relationship
between the size of the Grundy function obtained for $\sz$ and the size of the
Grundy functions of the factors $\al_v$.  The most significative results on Grundy
Functions are proved for \sgf, so we obtain a wide generalization of this results.

\section{Preliminaries.}
For general concepts we refer the reader to \cite{cB85}.\\
Let $D$ a digraph; $V(D)$ and $A(D)$ will denote the sets of vertices and arcs of $D$
respectively. Let $X_1, X_2$ be a subset of $V(D)$; the arc $(u_1,u_2)$ will be called an
$X_1X_2$-arc whenever $u_1\in X_1$ and $u_2\in X_2$. And $D[X_1]$ will denote the subdigraph
of $D$ induced by $X_1$. 

A set $I\subseteq V(D)$ is independent if $A(D[I])=\emptyset$. A semi-kernel $S$ of $D$ is
an independent set such that if there exist a $Sz$-arc then must exist a $zS$-arc in
$D$. A kernel $N$ of $D$ is an independent set of vertices such that for each
$z\in V(D)-N$ there exists a $zN$-arc. A digraph $D$ is a kernel-perfect digraph whenever
each one of its induced subdigraphs has a kernel.

A non-negative function $g$ is called a Grundy function on $D$ if for every vertex,
$x$, $g(x)$ is the smallest non-negative integer which does not belong to the set
$\{g(y)|y\in\Gamma^+(x)\}$.

The concepts of semi-kernel, kernel and Grundy functions of a digraph are nearly related
as we can see in Theorem \ref{teo1},  Theorem \ref{teo2} and Theorem \ref{teo3}.

In \cite{cB85} C. Berge defined the cartesian sum of $n$ digraphs
$D_1, D_2,\ldots, D_n$, denoted by $D_1+D_2+\ldots +D_n$, as follows:
\begin{enumerate}[i)]
\item $\displaystyle{V(D_1+D_2+\ldots +D_n)=\prod_{i=1}^{n}V(D_i)}$,
\item $\displaystyle{\Gamma(x_1,x_2,...x_n)=\bigcup_{i=1}^{n}\Big(\{x_1\}\times,...,\times\{x_{i-1}\}\times
  \Gamma(x_i)\times,...,\times \{x_n\}\Big)}$
\end{enumerate}
This operation comes naturally from the theory of Games. In \cite{cB85}, C. Berge proved:
\begin{theorem}\label{teo_berge}
  The cartesian sum $D_1+D_2+\ldots +D_n$ of digraphs having Grundy function, also has a
  Grundy function.
\end{theorem}

Let $D$ a digraph, $\al=(\al_v)_{v\in V(D)}$ a family where the $\al_v$ are mutually
disjoint digraphs. The Cartesian product of $\al$ over $D$, denoted by $\sz$ is defined as
follows:
\begin{enumerate}[i)]
\item $V(\sz)=\bigcup_{v\in V(D)}V(\al_v)$,
\item $A(\sz)=\big(\bigcup_{v\in V(D)}A(\al_v)\big)\cup \{(x,y)|x\in V(\al_u), y\in
  V(\al_v) \text{ and } (u,v)\in A(D)\}$.
  \end{enumerate}

\section{Semi-Grundy functions.}

\begin{definition}
  Let $D$ be a digraph. A function $s:V(D)\to \na$ is a \emph{semi-Grundy} function if
  satisfies:
  \begin{enumerate}
  \item $s(x)=k$ implies that each $y\in \Gamma^+(x)$ satisfies
    $s(y)\not=k$; \label{def:inciso1}
  \item $s(x)=k$ and $y\in \Gamma^+(x)$ with $s(y)>k$ implies that there exists
    $z\in \Gamma^+(y)$ such that $s(z)=k$. \label{def:inciso2}
  \end{enumerate}
\end{definition}

\begin{remark}
  Every Grundy function is in particular a semi-Grundy function. This tell us that the
  following families of digraphs has semi-Grundy functions: acyclic, transitive,
  kernel-perfect and digraphs without odd-cycles.
\end{remark}

\begin{lemma}
  If $D$ has a \sgf then $D$ has a semi-kernel. 
\end{lemma}
\begin{proof}
  Let $m_0:=\min \{s(x)|x\in V(D)\}$ and we define $S:=s^{-1}(m_0)$. We will prove that
  $S$ is a semi-kernel of $D$.

  Let $x, y \in S$, this implies that $s(x)=m_0=s(y)$. From \ref{def:inciso1} of the
  definition of \sgf this implies that $x\not\in\Gamma^+(y)$ and $y\not\in\Gamma^+(x)$. Thus $S$
  is an independent set.

  Now suppose that $x\in S$ and $y\in\Gamma^+(x)$. Thus $s(x_0)=m_0$ and from the
  definition of $m_0$ it follows that $s(y)>m_0$. From \ref{def:inciso2} of the definition
  of \sgf, there exists $z\in \Gamma^+(y)$ such that $s(z)=m_0$. Then $S$ is a semi-kernel
  of $D$.
  \end{proof}

\begin{remark}
  This concept differs from the concept of \emph{pseudo-Grundy} function presented in
  \cite{cB85}: a pseudo-Grundy function determines a kernel in a digraph.
\end{remark}

Now we point out some interesting facts between the concepts of semi-kernel, kernel, \gf
 and \sgf.

\begin{itemize}
\item A digraph can have more than one \sgf (see Figure \ref{f1}).
\item A digraph can have a \gf $f$ and \sgf $s$, where the maximum value of $f$ is greater
  than the maximum value of $s$ (see Figure \ref{f2}).
\item A digraph can have a \gf $f$ and \sgf $s$, where the maximum value of $f$ is less
  than the maximum value of $s$ (see Figure \ref{f3}).
\item If a digraph have \sk, this not implies that it has \sgf (see Figure \ref{f4}).
\item There are digraphs with \sgf and no \gf (see Figure \ref{f5}).
\item If a digraph has a \sgf this not implies that it has a kernel (see Figure \ref{f6}).
\item Having \sgf is not an hereditary property on induced subdigraphs (see Figure \ref{f6}). 
\end{itemize}

\begin{figure}[ht]
  \centering
  \includegraphics[width=70mm]{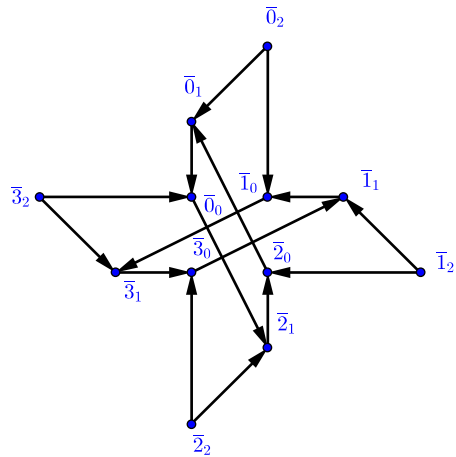}    
  \caption{A digraph with two different Grundy functions: one of them is given by the
    numbers $\bar{x}$, the other by the subscripts.}
  \label{f1}
\end{figure}

\newpage

\begin{figure}[ht]
  \centering
  \includegraphics[width=45mm]{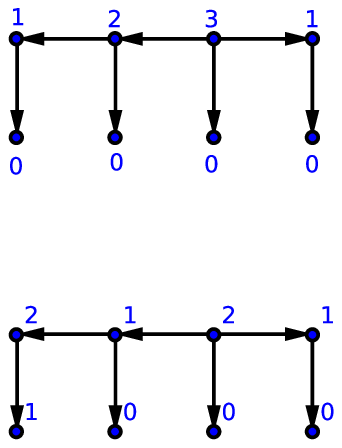}
  \caption{A digraph with a  Grundy functions and a \sgf}
  \label{f2}
\end{figure}

\begin{figure}[ht]
  \centering
  \includegraphics[width=38mm]{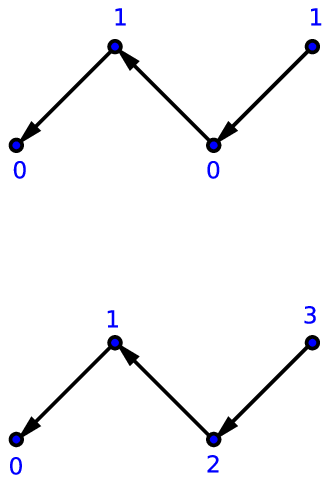}
  \caption{A digraph with a  Grundy functions and a \sgf}
  \label{f3}
\end{figure}

\begin{figure}[ht]
  \centering
  \includegraphics[width=60mm]{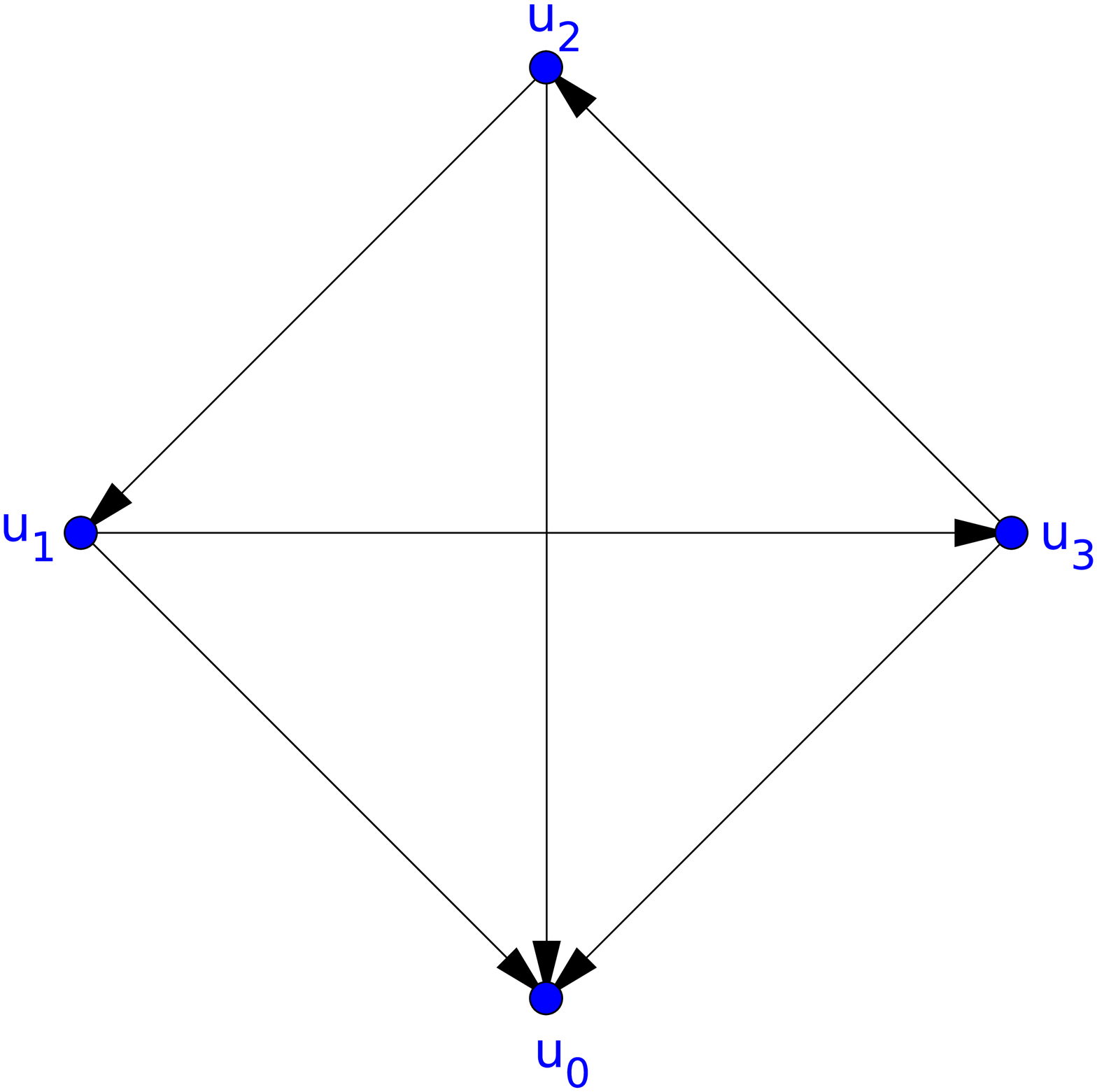}
  \caption{A digraph with semi-kernel and no \sgf}
  \label{f4}
\end{figure}

\begin{figure}[ht]
  \centering
  \includegraphics[width=48mm]{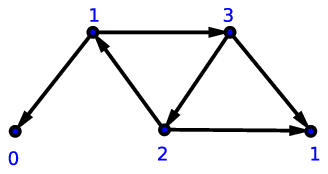}
  \caption{A digraph with \sgf and no Grundy function}
  \label{f5}
\end{figure}

\begin{figure}[ht]
  \centering
  \includegraphics[width=45mm]{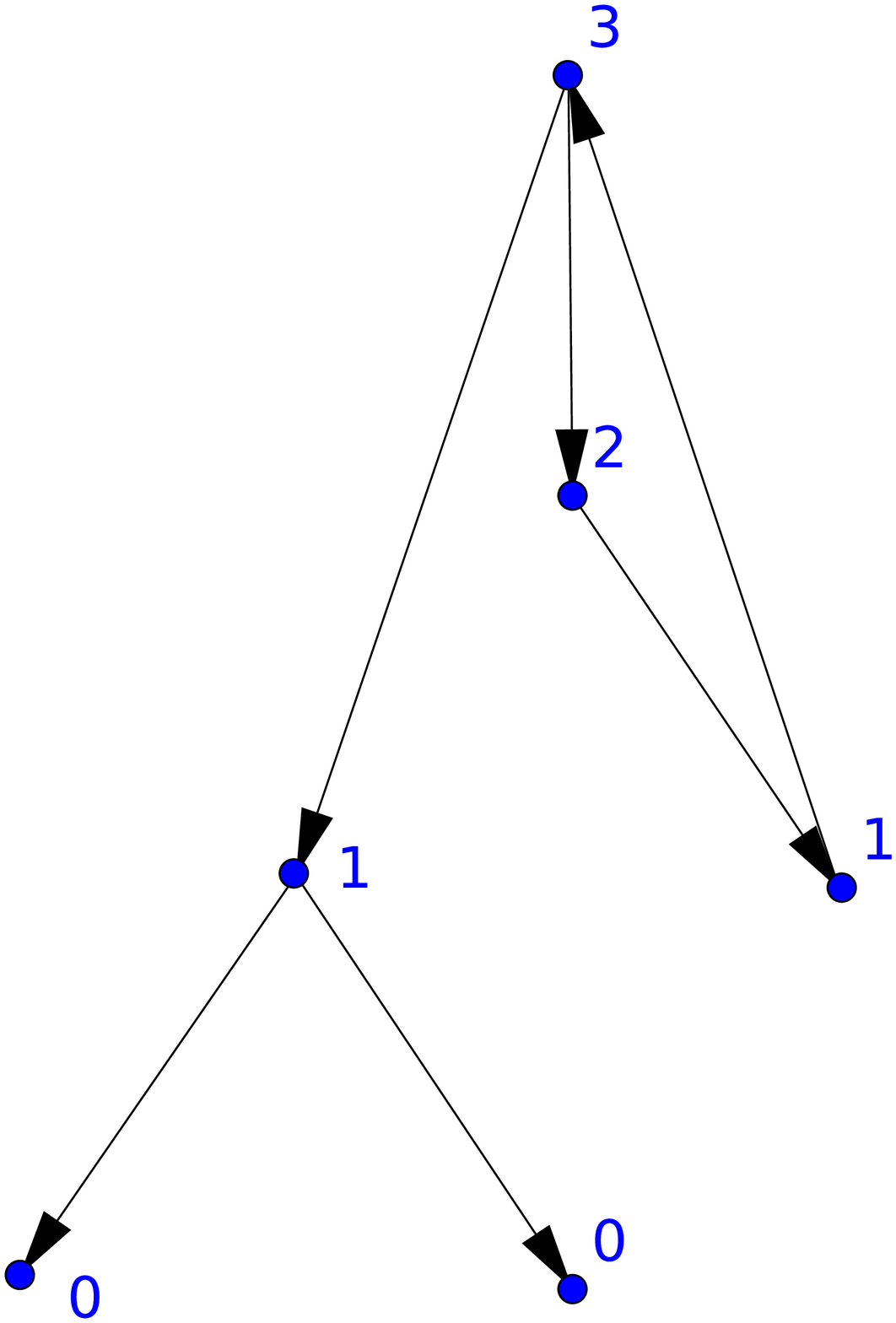}
  \caption{A digraph with \sgf and no kernel}
  \label{f6}
\end{figure}

\newpage

A natural generalization of Theorem 3 is the following result:
\begin{proposition}
  If every induced subdigraph of $D$ has a semi-kernel then $D$ has a \sgf.
\end{proposition}
\begin{proof}
  Let $D$ be such a digraph and let $s_0$ a semi-kernel of $D$. Let $S_1$ a semi-kernel of
  $D_1:=D[V(D)\setminus S_0]$. Let $S_2$ be a semi-kernel of $D_2:=D[V(D)\setminus
  (S_0\cup S_1)]$, etc. The sets $S_i$ form a partition of $V(D)$. We define an integer
  function $s(x)$ on $V(D)$ by:
  \[
    s(x)=k \text{ if and only if } x\in S_k
  \]

  Now we show that $s$ is a semi-Grundy function:
  \begin{enumerate}
  \item Let $x\in V(D)$ with $s(x)=k$ and $y\in\Gamma^+(x)$. We must show that
    $s(y)\not=k$.
    From the definition of $s$ we know that $x\in S_k$ and that $S_k$ is a semi-kernel of
    $D_k$. Since $(x,y)\in A(D)$ then $y\not\in S_k$ and $s(y)\not=k$.
  \item Suppose that $s(x)=k$, $y\in\Gamma^+(x)$ and $s(y)>k$. We must show that there
    exists $z\in\Gamma^+(y)$ such that $s(z)=k$.

    Since $s(x)=k$ we know that $x\in S_k$ with $S_k$ a semi-kernel of $D_k$. By
    hypothesis we have that $s(y)>k$, this tell us that $y\in V(D_k)$. Since $(x,y)\in
    A(D_k)$ and $S_k$ is a semi-kernel of $D_k$ it follows that exists $z\in\Gamma^+(y)$
    with $z\in S_k$. Then, by the definition of $s$, $s(z)=k$.         
  \end{enumerate}

  Then $s$ is a semi-Grundy function on $D$.
\end{proof}

As a generalization of  Theorem \ref{teo_berge}, we can prove that the cartesian sum of digraphs
having \sgf, also has \sgf. There is an important difference in the proof of this theorem; in \cite{cB85} for the proof
of Theorem \ref{teo_berge} C. Berge use the digital sum. In our proof, we do not use the
digital sum, just the ordinary sum of the integers.

\begin{proposition}
  Let $D_i$ be  digraphs, for $i=1,\ldots, n$. If $D_i$ has a \sgf $s_i$ for $i=1,\ldots, n$ then the function
  \[
    \displaystyle{S(x_1,\ldots,x_n):=\sum_{i=1}^ns_i(x_i)}
  \]
  is a \sgf of $D:=D_1+D_2+\ldots+D_n$. The maximum value of $S$ is
  $\displaystyle{\sum_{i=1}m_i}$, where $m_i$ is the maximum value of $s_i$.
\end{proposition}
\begin{proof} 
  Let $D_i$ and $s_i$ be as in the hypothesis. We will prove that $S$ is a \sgf of $D:=D_1+D_2+\ldots+D_n$. 
  \begin{enumerate}
  \item Let $x=(x_1,\ldots,x_j,\ldots,x_n)\in D$ such that
    $S(x)=s_1(x_1)+\ldots+s_n(x_n)=k$ and $y\in \Gamma^+(x)$. We must show that $S(y)\not=k$.
    
    Since $y\in \Gamma^+(x)$, from the definition of $D$, we have that
    $y=(x_1,\ldots,x'_j,\ldots,x_n)$ for some $j\in \{1,\ldots,n\}$, where
    $(x_j,x'_j)\in A(D_j)$. Since $s_j$ is a \sgf of $D_j$, $s_j(x_j)\not=s_j(x'_j)$ and
    we have:
    \begin{equation*}
      \begin{array}{rcl}
        S(y) & = & S(x_1,\ldots,x'_j,\ldots,x_n)\\
        & = & s_1(x_1)+\ldots+s_j(x'_j)+\ldots+s_n(x_n)\\
        & \not= & k. 
        \end{array}
    \end{equation*}    
  \item Suppose that $S(x)=k$, $y\in\Gamma^+(x)$ and $S(y)>k$. We must show that there
    exists $z\in\Gamma^+(y)$ such that $S(z)=k$.\\
    Let $x=(x_1,\ldots,x_j,\ldots,x_n)\in D$ and $y\in\Gamma^+(x)$, then from the
    definition of $D$ we know that $y=(x_1,\ldots,x'_j,\ldots,x_n)$. 
    From the definition of $S$, we have:
    \begin{equation*}
      \begin{array}{rcl}
        S(y) & = & S(x_1,\ldots,x'_j,\ldots,x_n)\\
        & = & s_1(x_1)+\ldots+s_j(x'_j)+\ldots+s_n(x_n)\\
        & > & k (\text{ hypothesis})
        \end{array}
    \end{equation*}    
    This implies that $s_j(x_j)<s_j(x'_j)$ and since $s_j$ is a \sgf of $D_j$, there
    exists $x''_j\in \Gamma^+_{D_j}(x'_j)$ such that $s_j(x''_j)=s_j(x_j)$. Then the
    vertex $z=(x_1,\ldots,x''_j,\ldots,x_n)$ is in $y\in\Gamma^+(x)$ and 
    \begin{equation*}
      \begin{array}{rcl}
        S(z) & = & S(x_1,\ldots,x''_j,\ldots,x_n)\\
             & = & s_1(x_1)+\ldots+s_j(x''_j)+\ldots+s_n(x_n)\\
             & = & s_1(x_1)+\ldots+s_j(x_j)+\ldots+s_n(x_n)\\
             & = & k 
        \end{array}
    \end{equation*}        
  \end{enumerate}
  Then $S$ is a \sgf of $D$.
\end{proof}

\begin{remark}
  Now consider that $D$ is a digraph and has a \sgf $S'$. From the definition of \sgf its
  minimum value could be any integer. It would be easy to work with, if we know which is
  its minimum value, as happen with the Grundy functions, and which is next value, and so
  on.  With this purpose, we define a new \sgf, $S$, in the following way: Suppose that the
  image of $S'$ is $\{m_0, m_1,\ldots,m_r\}$ where $m_i<m_j$ if and only if
  $i<j$. Then:
  \begin{center}
    $ S:V(D)\to \na$,\\ $S(x)=k$ if and only if $S'(x)=m_k$.
  \end{center}
  It is easy to see that this a \sgf of $D$ and induced the same partition in independent
  subsets of $V(D)$ that $S'$. This \sgf starts in $0$ and take consecutive positive integer
  values. So, from now on, we will assume that every \sgf take consecutive non-negative
  integer values starting from $0$.
\end{remark}

\section{\sgf on the cartesian product.}

The cartesian product is an operation on digraphs that comes naturally
from the Theory of Games. In \cite{hGrG10}, we studied the behavior of the \gf and the
cartesian product, and in some cases we found some bounds of Grundy functions. In this
section we generalizes those results to the concept of \sgf.

\begin{theorem}\label{teo:1}
  Let $D$ be a digraph and $\alpha=(\alpha_v)_{v\in V(D)}$ a family of mutually disjoint
  digraphs. If $D$ is a kernel-perfect digraph and each $\al_v$ has a semi-Grundy function, then
  $\sz$ possesses a semi-Grundy function. 
\end{theorem}
\begin{proof}
  Let $D$ and $\al=(\alv)_{v\in V(D)}$ as in the hypothesis. We consider for each
  $v\in V(D)$ any fixed semi-Grundy function $f_v$ of $\alv$ and $S_0$ a kernel of $D$.
  Now we define the following sets:\\
  $N_0:=\{x\in V(\sz)|f_y(x)=0 \text{ for some } y\in S_0 \}$.\\
  $M_0:=\{y\in V(D)| V(\al_y)\subseteq N_0\}$.\\
  Let $S_1$ a kernel of $D_1:=D[V(D)\setminus M_0]$ (the subdigraph of $D$ induced by
  $V(D)\setminus M_0$). For $y\in S_1$ we denote: $m(1,y):=\min \{f_y(x)|x\in
  V(\al_y)\setminus N_0\}$.\\
  $N_1:=\{x\in \big(V(\sz)\setminus N_0\big)|f_y(x)=m(1,y) \text{ for some } y\in S_1 \}$.\\
  $M_1:=\{y\in V(D)| V(\al_y)\subseteq (N_0\cup N_1)\}$.\\(see figure \ref{suma_cartesiana})

\begin{center}
  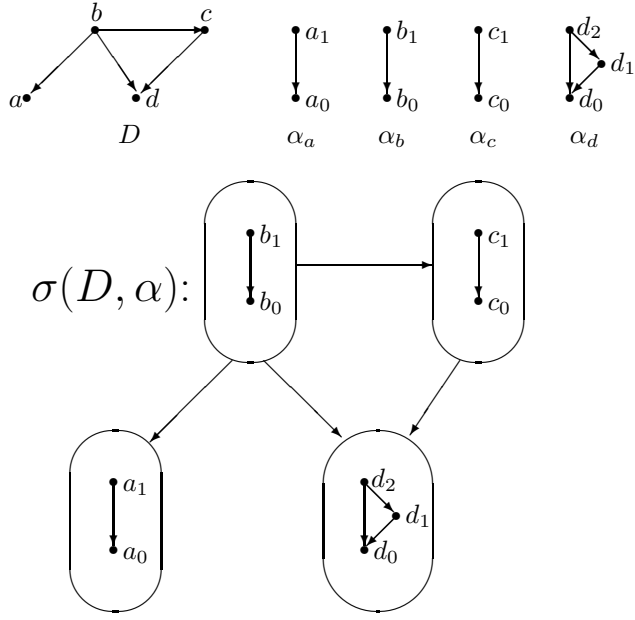
\begin{figure}[ht]
    \setlength{\unitlength}{0.6mm}
    \begin{picture}(120,150)
      \put(52,90){           
        \put(-63,35){\circle*{2}}      
        \put(-39,35){\circle*{2}}
        \put(-48,50){\circle*{2}}
        \put(-24,50){\circle*{2}}
        \multiput(-48,50)(24,0){2}{\vector(-1,-1){14}} 
        \put(-48,50){\vector(1,0){24}}
        \put(-48,50){\vector(2,-3){9}} 
        \put(-43,25){$D$}
        \put(-67,33){$a$}
        \put(-49,52){$b$}
        \put(-25,52){$c$}
        \put(-37,33){$d$}
        
        \put(-4,50){\circle*{2}}
        \put(-4,35){\circle*{2}}
        \put(-4,50){\vector(0,-1){14}} 
        \put(-6,25){$\al_a$}
        \put(-2,48){$a_1$}
        \put(-2,33){$a_0$}
        
        \put(16,50){\circle*{2}}
        \put(16,35){\circle*{2}}
        \put(16,50){\vector(0,-1){14}}
        \put(14,25){$\al_b$}
        \put(18,48){$b_1$}
        \put(18,33){$b_0$}

        \put(36,50){\circle*{2}}
        \put(36,35){\circle*{2}}
        \put(36,50){\vector(0,-1){14}}
        \put(34,25){$\al_c$}
        \put(38,48){$c_1$}
        \put(38,33){$c_0$}
        
        \put(56,50){\circle*{2}}
        \put(56,35){\circle*{2}}
        \put(63,42.5){\circle*{2}}
        \put(56,50){\vector(0,-1){14}}              
        \put(56,50){\vector(1,-1){6.5}}
        \put(63,42.5){\vector(-1,-1){6.5}}
        \put(56,25){$\al_d$}
        \put(57.5,49){$d_2$}
        \put(58,33){$d_0$}
        \put(65,41){$d_1$}

        \put(-40,-100){
          \put(-4,50){\circle*{2}}
          \put(-4,35){\circle*{2}}
          \put(-4,50){\vector(0,-1){14}} 
          \put(-2,48){$a_1$}
          \put(-2,33){$a_0$}             
        }
        
        \put(-30,-45){
          \put(16,50){\circle*{2}}
          \put(16,35){\circle*{2}}
          \put(16,50){\vector(0,-1){14}}
          \put(18,48){$b_1$}
          \put(18,33){$b_0$}
        }              
        
        \put(0,-45){
          \put(36,50){\circle*{2}}
          \put(36,35){\circle*{2}}
          \put(36,50){\vector(0,-1){14}}
          \put(38,48){$c_1$}
          \put(38,33){$c_0$}
        }

        \put(-45,-100){
          \put(56,50){\circle*{2}}
          \put(56,35){\circle*{2}}
          \put(63,42.5){\circle*{2}}
          \put(56,50){\vector(0,-1){14}}              
          \put(56,50){\vector(1,-1){6.5}}
          \put(63,42.5){\vector(-1,-1){6.5}}
          \put(57.5,49){$d_2$}
          \put(58,33){$d_0$}
          \put(65,41){$d_1$}
        }
        \put(-17,-5){
          \put(3,1.5){\oval(20,40)}
        }
        
        \put(33,-5){
          \put(3,1.5){\oval(20,40)}
        }

        \put(-46.5,-60){
          \put(3,1.5){\oval(20,40)}
        }

        \put(11,-60){
          \put(3,1.5){\oval(24,40)}
        }
        \put(-4,-2){\vector(1,0){30}}
        \put(-11,-23){\vector(1,-1){17}}
        \put(-18,-23){\vector(-1,-1){18}}
        \put(32,-23){\vector(-2,-3){11}}
        \put(-62,-10){\huge{$\sz$:}}
        
      }
    \end{picture}
    \caption{In this example, the Grundy function of the digraphs $\al_w$ are
      given by the subscript of each vertex and the arcs between the ovals means
      that there must be an arc from any vertex of the source to any vertex in the
      target: $S_0=\{a,d\}, N_0=\{a_1,d_1\}, M_0=\emptyset$;
      $S_1=\{a,d\}, m(1,a)=1, m(1,d)=1, N_1=\{a_2,d_2\}, M_1=\{a\}; S_2=\{d\},
      m(2,d)=2, N_2=\{d_2\}$...}
    \label{suma_cartesiana}
  \end{figure}    
\end{center}
  
     Clearly $N_0\cap N_1=\emptyset$, $N_0\not=\emptyset$ and $N_1\not=\emptyset$.\\
     Continuing this way we define a sequence of subsets of vertices of $\sz$ and
     $D$ as follows: if $D_i, N_i$ and $M_i$ are defined and $S_i$ is a kernel of
     $D_i$ then we
     defined $D_{i+1}, S_{i+1}, N_{i+1}$ and $M_{i+1}$ as follows:\\
     $D_{i+1}:=D[V(D)\setminus\big(M_0\cup\ldots\cup M_i\big)]; S_{i+1}$ is a kernel
     of
     $D_{i+1}$;\\
     $m(i+1,y):=\min\{f_y(x)|x\in \big(V(\al_y\setminus (N_0\cup\ldots\cup
     N_i)\big)\}$ for
     $y\in S_{i+1}$;\\
     $N_{i+1}:=\{x\in \big(V(\sz)\setminus (N_0\cup\ldots\cup
     N_i)\big)|f_y(x)=m(i+1,y) \text{ for some } y\in S_{i+1}\}$.
     $M_{i+1}:=\{y\in V(D)|V(\al_y)\subseteq (N_0\cup\ldots \cup N_{i+1})\}$.

  Clearly $N_i\cap N_j=\emptyset$ for any $i, j $ with $i\not= j$. This procedure finishes
  when we get the first natural number $r$ such that $V(D_r)=\emptyset$. Notice that this
  natural number $r$ exists as $N_i\cap N_j=\emptyset$ whenever $i\not= j$ and
  $N_i\not=\emptyset$ for each $0\leq i\leq r-1$.\\
  Now we define the function $S:V(\sz)\to \na$ as follows:
  \[
    S(x)=k \text{ if and only if } x\in N_k.
  \]
  $S$ is well defined as $N_i\cap N_j=\emptyset$ for any $i, j $ with $i\not= j$ and
  $V(D_r)=\emptyset$. 
  We will prove that $S$ is a semi-Grundy function on $\sz$.  
  \begin{enumerate}
   \item Let $x\in V(\sz)$ such that $S(x)=k$ and $y\in\Gamma^+(x)$. We must show that
    $S(y)\not=k$.\\
    By definition of $\sz$ there exists  vertices $u, v\in V(D)$ such that $x\in V(\al_u)$
    and $y\in V(\al_v)$.
    Suppose that $S(x)=S(y)=k$. If $u=v$ this implies that $f_u(x)=f_u(y)$ and then
    $y\not\in \Gamma^+(x)$ (because $f_u^{-1}(k)$ is an independent set). This contradicts
    the assumption $y\in\Gamma^+(x)$ then, in this case, $S(y)\not=k$.

    Now suppose that $S(x)=S(y)=k$ and $u\not=v$. From the definition of the sets $S_k$,
    we know $S_k$ is a semi-kernel of $D_k$ and that $u, v\in S_k$. Since $S_k$ is an
    independent subset of $V(D)$ it follows that there is no arc between $u$ and $v$ in
    $D$, and then by definition of $\sz$ there is no arc between any vertex from $\al_u$
    to any vertex of $\al_v$ or viceversa. In particular, there is no arc between $x$ and
    $y$. This contradicts the hypothesis that $y\in\Gamma^+(x)$. Thus in this case also it
    satisfies $S(y)\not=k$.
    
  \item Suppose that $S(x)=k$ and $y\in\Gamma^+(x)$ with $S(y)>k$. We must show that there
    exists $z\in\Gamma^+(y)$ such that $S(z)=k$.

    By definition of $\sz$ there exists $u, v\in V(D)$ such that $x\in V(\al_u)$ and
    $y\in V(\al_v)$. First we prove the case when $u=v$.  In this case, $S(x)=k<S(y)$
    implies that $f_u(x)<f_u(y)$ (because $f_u$ is a \sgf) and then, must exists
    $z\in \Gamma^+_{\al_u}(y)$ such that $f_u(z)=f_u(x)$. From the definition of
    $\sz$ and $S$, we have $S(z)=k=S(x)$ and $z\in \Gamma^+(y)$.
     
    When $u\not=v$, by definition of $\sz$ and the fact there is $xy$ arc in $\sz$,
    there is a $uv$ arc in $D$. Since $S(x)=k<S(y)$, we know that $u$ is an element of
    $S_k$ that is a kernel of $D_k$ and $v\in V(D_k)$. Then must exists $w\in S_k$ such
    that $(v,w)\in A(D_k)$. From this and the definition of $S$, must exists $z\in
    V(\al_w)$ such that $S(z)=k$ and $(y,z)\in A(\sz)$.    
  \end{enumerate}
  Then $S$ is a \sgf of $\sz$.
\end{proof}

\begin{coro}\label{teo:2}
  Let $D$ be a digraph and $\alpha=(\alpha_v)_{v\in V(D)}$ a family of mutually disjoint
  digraphs. If $D$ is a kernel-perfect digraph and each $f_u$ has a semi-Grundy
  function, then $\sz$ possesses a semi-Grundy function $S$ such that
  $\displaystyle{\max \{S(x)|x\in V(\sz)\}\leq \sum_{u\in V(D)}m_u+|V(D)|-1}$, where
  $m_u=\max \{f_u(x)|x\in V(\al_u)\}$.
\end{coro}
\begin{proof}
  Since $D$ is a kernel-perfect digraph then $D$ has a \sgf $f$ (indeed is a Grundy
  function). Let $S$ be the \sgf of $\sz$ defined in the proof of Theorem
  \ref{teo:1}. Since $S^{-1}(i)=\cup f^{-1}(j)$ for some $j\in\na$ and $v\in V(D)$ we have
  that the bound reaches whenever $S^{-1}(i)=f^{-1}(j)$ for some $j\in\na$ and $v\in
  V(D)$. 
\end{proof} 

\begin{theorem} 
  Let $D$ be a digraph and $\alpha=(\alpha_v)_{v\in V(D)}$ a family of mutually disjoint
  digraphs. If $\sz$ has a \sgf then $D$ has a semi-kernel and $\al_v$ has a \sgf for each
  $v\in V(D)$.   
\end{theorem}
\begin{proof}
  First, we prove that $D$ has a \sk. Let $S$ be a \sgf of $\sz$; $N:=\{x\in
  V(\sz)|S(x)=0\}$ and $A:=\{u\in V(D)|V(\al_u)\cap N\not=\emptyset\}$.

  {\bf{$A$ is a \sk of $D$}}:\\
  Let $u, v\in A$, with $u\not=v$. Then there exists $x\in V(\al_u)$ and $w\in
  V(\al_v)$ such that $S(x)=0=S(w)$. Since $S$ is a \sgf in $\sz$ it follows that there is
  no arc in $\sz$ between $x$ and $w$. Hence from the definition of $\sz$,
  there is no arc between $u$ and $v$ in $D$. Then $A$ is an independent set of $D$. \\
  Let $u \in A$ and $w\in (V(D)\setminus A)$ such that $(u,w)\in A(D)$. Let $x\in
  V(\al_u)$ such that $S(x)=0$ and $y\in V(\al_w)$. From the definition of $\sz$ and since
  $S$ is a \sgf it follows that $S(y)>0$. Then, there exists $z\in \Gamma^{+}(y)$ such
  that $S(z)=0$. From the definition of $\sz$, there exists $v\in V(D)$ such that $z\in
  V(\al_v)$;  then $v\in A$, and  $(w,v)\in A(D)$. Thus $A$ is a \sk of $D$.

  Now let $u\in V(D)$ be. We will prove that $\al_u$ has a \sgf.  Since $S$ is a \gf
  of $\sz$, there exists a subset of integers
  $\{i_0,i_1,\ldots,i_r\}$ such that $S^{-1}(i_j)\cap V(\al_u)\not=\emptyset$, with
  $i_j<i_{j+1}$, for every $j=0,1,...,r-1$ and such that
  $\displaystyle{V(\al_u)\subseteq \bigcup_{j\in
      \{i_0,i_1,\ldots,i_r\}}S^{-1}(j)}$. Consider the function $s_u:V(\al_u)\to \na$
  given by $s_u(x)=j$ if and only if $x\in S^{-1}(j)$. By definition of $s_u$
  and since $S$ is a \sgf, follows that $s_u$ is well defined. Now we need to prove
  that is a \sgf of $\al_u$.

  \begin{enumerate}[1)]
  \item Let $x \in V(\al_u)$ be such that $s_u(x) = j$ and $y\in
    \Gamma_{\al_u}^{+}(x)$. Since $(x,y)\in A(\sz)$ then $S(x)\not=S(y)$, i.e,
    $s_u(x)=S(x)\not=S(y)=s_u(y)$. 
  \item Suppose that $s_u(x)=k$ and $y\in \Gamma_{\al_u}^{+}(x)$ with $s_u(y) > k$. We must show that there
    exists $z \in\Gamma_{\al_u}^{+}(y)$ such that $s_u(z)=k$.
    Since $S$ is a \sgf of $\sz$, there exists $z\in\Gamma_{\sz}^{+}(y)$ such that
    $S(z)=k$. Now, $z\in V(\al_u)$, otherwise $(y,z)\in A(\sz)$ would implied that
    $(x,z)\in A(\sz)$ and $S(x)=S(z)$, and so, $S$ would not be a \sgf of
    $\sz$. Thus, $z\in V(\al_u)$ and $s_u(z)=S(z)=k$.    
  \end{enumerate}
  Then, $s_u$ is a \sgf of $\al_u$.  
\end{proof}

\begin{theorem}
  Let $D$ be a digraph, $\al=(\al_u)_{u\in V(D)}$ a family of mutually disjoint
  digraphs. If $D$ has a \sgf $f$ and each $\al_u$ has a \sgf $s_u$ $(u\in
  V(D))$ such that $\max\{s_u(x)|x\in V(\al_u)\}=\max\{s_v(x)|x\in V(\al_u)\}$ whenever
  $\{u,v\}\subseteq f^{-1}(i)$ for some $i\in \na$. Then  $\sz$ possesses a \sgf $S$ which
  satisfies: $\max\{S(x)|x\in V(\sz)\}=n+\sum_{i=0}^{n}m_i$, where $n=\max\{f(x)|x\in
  V(D)\}$ and $m_i=\max\{s_u(x)|x\in V(\al_v), \text{ and } v\in f^{-1}(i)\}$, for each
  $0\leq i\leq n$.
\end{theorem}
\begin{proof}
  Define $S:V(\sz)\to\na$ as follows:\\
  Let $x\in V(\sz)$ be; we have that there exist a unique $u\in V(D)$ such that $x\in
  V(\al_u)$ and a unique $i\in \{0,1,\ldots,n\}$ such that $u\in f^{-1}(i)$. Then we define:
  \[
    S(x)=\sum_{j=0}^{i-1}m_j+f(u)+s_u(x);    
  \]
  whenever $1\leq i\leq n$ and $S(x)=s_u(x)$ whenever $i=0$. 
  
  We will prove that $S$ is \sgf of $\sz$.
  
  \begin{enumerate}
  \item Let $x\in V(\sz))$ such that $S(x)=k$ and $y\in \Gamma^{+}(x)$. We must show that
    $S(y)\not=k$.
    
    By definition of $\sz$ exist vertices $u, v \in V(D)$ such that
    $x\in V(\al_u)$ and $y\in V(\al_v)$. We will consider two cases:
    $u=v$ and $u\not=v$. If $u=v$ then we have that $x,y\in
    V(\al_u)$. Since $y\in \Gamma^{+}(x)$ and the definition of $\sz$
    follows that $y\in \Gamma^{+}_{\al_u}(x)$. Then,
    $s_u(x)\not=s_u(y)$ because $s_u$ is a \sgf of $\al_u$. Thus:
    \[
      S(x)=\sum_{j=0}^{i-1}m_j+f(u)+s_u(x)\not=\sum_{j=0}^{i-1}m_j+f(u)+s_u(y)=S(y);        
    \]
    (assuming that $f(u)=i$).
    
    Now suppose that $u\not=v$. In this case, from $y\in \Gamma^{+}(x)$ and the definition
    of $\sz$ we conclude that $(u,v)\in A(D)$, and since $f$ is a \sgf $f(u)\not=f(v)$.
    If $f(u)=i<k=f(v)$ we have:
    \begin{equation*}    
      \begin{array}{rcccl}
        S(x) & = & \displaystyle{\sum_{j=0}^{i-1}m_j+f(u)+s_u(x)}&\leq & \displaystyle{\sum_{j=0}^{i-1}m_j+m_i+f(u)}\\
             & < & \displaystyle{\sum_{j=0}^{i-1}m_j+m_i+f(v)}&\leq & \displaystyle{\sum_{j=0}^{k-1}m_j+f(v)}\\
             & \leq & \displaystyle{\sum_{j=0}^{k-1}m_j+f(v)+s_v(y)}& = & S(y)\\           
      \end{array}                 
    \end{equation*}
    Then $S(x)<S(y)$. If $f(u)>f(v)$ the argument is similar. Then in either case, $S(x)\not=S(y)$.   
    
  \item Suppose that $S(x)=k$ and $y\in\Gamma^+(x)$ with $S(y)>k$. We must show that there
    exist $z\in\Gamma^+(y)$ with $S(z)=k$.
    
    Let $u, v\in V(D)$ such that $x\in V(\al_u)$ and $y\in V(\al_v)$. If $u=v$ then from
    the definition of $S$, follows that $s_u(x)<s_u(y)$ and since $s_u$ is a \sgf of
    $\al_u$, must exist $z \in \Gamma^+_{\al_u}(y)$ (and then $z \in \Gamma^+(y)$) such
    that $s_u(z)=s_u(x)$. Then we have that:
    $S(x)=\sum_{j=0}^{i-1}m_j+f(u)+s_u(x)=\sum_{j=0}^{i-1}m_j+f(u)+s_u(z)=S(z)$, thus
    $S(z)=k$.
    
    If $u\not=v$ then $f(u)\not=f(v)$. From the definition of $S$ and the fact that
    $S(x)<S(y)$ follows $f(u)<f(v)$. Since $f$ is a \sgf there exists $w\in
    V(D)$ with $w\in \Gamma^+_D(v)$ and $f(w)=f(u)$. By the hypothesis, we know that
    $\max\{s_u(x)|x\in V(\al_u)\}=\max\{s_w(x)|x\in V(\al_w)\}$, and then there exist
    $z\in V(\al_w)$ with $s_w(z)=s_u(x)$. Then we have:\\ 
    $S(x)=\sum_{j=0}^{i-1}m_j+f(u)+s_u(x)=\sum_{j=0}^{i-1}m_j+f(w)+s_w(z)=S(z)$, thus $S(z)=k$.    
  \end{enumerate}
  Thus $S$ is a \sgf of $\sz$.
\end{proof}

\section{Digraphs $R_n$}

Now we construct a digraph $R_n$, for all $n\geq 2$, with the property that has 2
different Grundy functions: $g^{(n)}_1$ and $g^{(n)}_2$ These functions satisfied:
$\max\{g^{(n)}_1(x)|x\in V(R_n)\}=1$ and $\max\{g^{(n)}_2(x)|x\in V(R_n)\}=n$.
These digraphs are constructed by recursion.

Let $R_2$ be the digraph in Figure \ref{f1}. We define two functions on $V(D)$:
\begin{equation*}  
g^{(2)}_1(\bar{x_p}) = \left\{
\begin{array}{rl}
0, & \text{ if } x+p\equiv 0 \mod 2\\
1, & \text{ if } x+p\equiv 1 \mod 2
\end{array} \right.
\end{equation*}
 and $g^{(2)}_2(\bar{x}_p)=p$. It is easy to see that this functions are \gf of $R_2$. Now we
 construct $R_3$, $R_4$, and so on, by recursion.

Suppose that the digraph $R_n$ is defined and has two Grundy functions $g^{(n)}_1, g^{(n)}_2$ defined
by:
\begin{equation*}  
g^{(n)}_1(\bar{x_p}) = \left\{
\begin{array}{rl}
0, & \text{ if } x+p\equiv 0 \mod 2\\
1, & \text{ if } x+p\equiv 1 \mod 2
\end{array} \right.
\end{equation*}
 and $g^{(n)}_2(\bar{x}_p)=p$. 

 Now, we define the digraph $R_{n+1}$ as
 follows (see Figure \ref{R_3}): 
\begin{itemize}
\item $V(R_{n+1})=V(R_n)\cup \{\bar{0}_{n+1},\bar{1}_{n+1},\bar{2}_{n+1},\bar{3}_{n+1},\}$;
\item $\displaystyle{A(R_{n+1})=A(R_n)\cup \bigcup_{\bar{x}\in\{\bar{0},\bar{1},\bar{2},\bar{3}\}}A_{\bar{x}}}$;\\
    where
    $A_{\bar{x}}=\{(\bar{x}_{n+1},\bar{x}_{n-2i})|0\leq n-2i\leq n, i\in\na\}\cup\\
    \{(\bar{x}_{n+1},\overline{(x+1)}_{n-(2i+1)})|0\leq n-(2i+1)\leq n-1,i\in\na\}$. 
\end{itemize}
Note that the sum in $\overline{(x+1)}$ is the sum of $\en_4$.

     \newpage
     
\begin{figure}[ht]
  \centering
  \includegraphics[width=140mm]{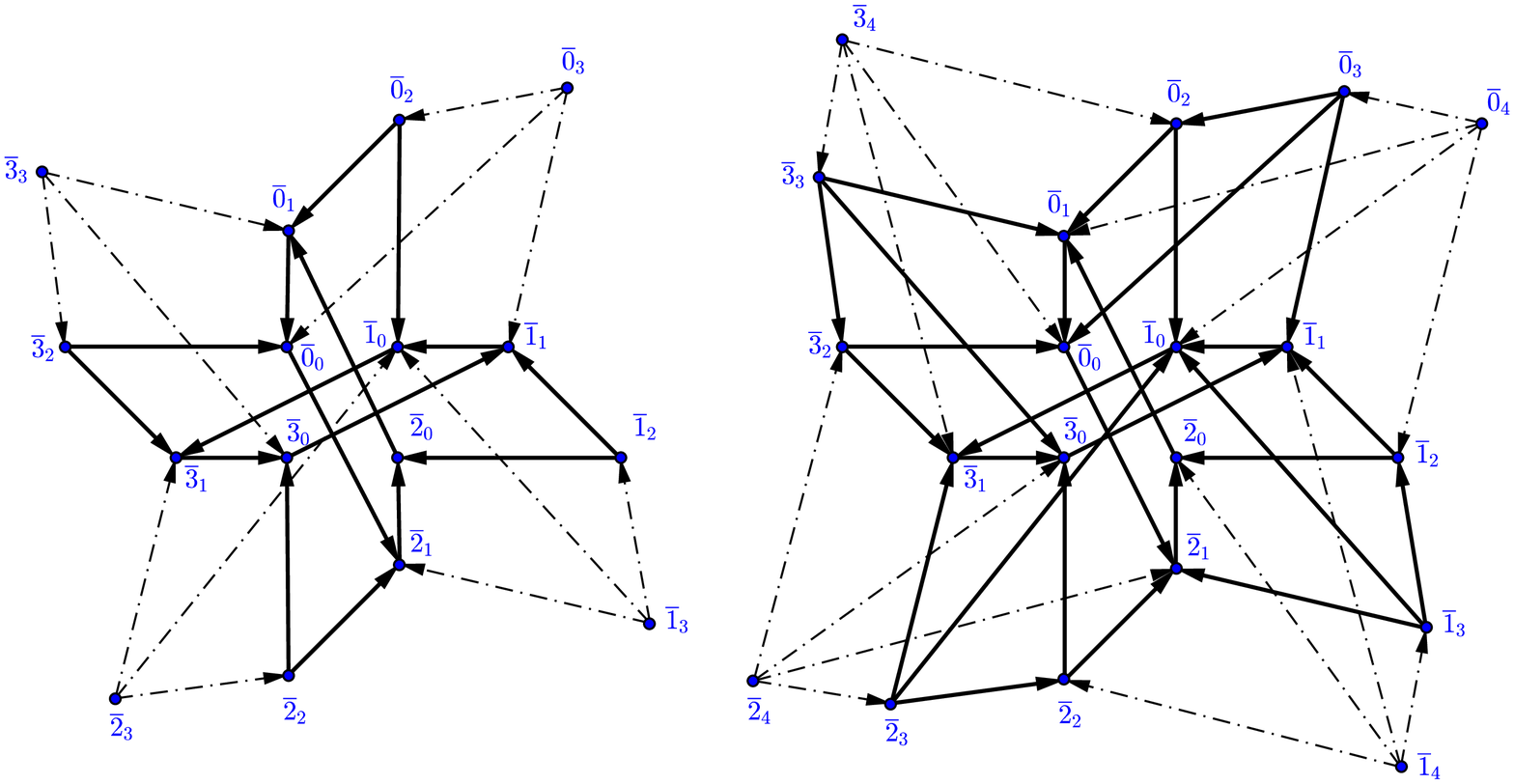}
  \caption{Digraphs $R_3$ constructed from $R_2$ and $R_4$ constructed from $R_3$}
  \label{R_3}
\end{figure}

Now we define two functions on $g^{(n+1)}_1,g^{(n+1)}_2:V(R_{n+1})\to \na$ as follows: 
\begin{equation*}  
g^{(n+1)}_1(\bar{x_p}) = \left\{
\begin{array}{rl}
0, & \text{ if } x+p\equiv 0 \mod 2\\
1, & \text{ if } x+p\equiv 1 \mod 2
\end{array} \right.
\end{equation*}
and $g^{(n+1)}_2(\bar{x}_p)=p$.
\begin{remark}
  Its important to note that the Grundy function $g^{(n)}_1$ only takes two values: $0$
  and $1$, and then, $R_n$ its a bipartite graph. 
\end{remark}

From the definition of both functions is easy to see that
$g^{(n)}_i(\bar{x}_p)=g^{(n+1)}_i(\bar{x}_p)$ if $p<n+1$, $i=1,2$. So, we just have to
  prove that $g^{(n+1)}_i$ extends $g^{(n)}_i$.

  Note that $\{\bar{0}_{n+1}, \bar{1}_{n+1}, \bar{2}_{n+1}, \bar{3}_{n+1}\}$ is an
  independent set of $R_{n+1}$. In $R_{n+1}$, $\bar{x}_{n+1}$ only is adjacent to
  $\bar{x}_{n-2i}$ or $\overline{x+1}_{n-(2i+1)}$. Since $x+(n+1)+x+n-2i=2x+2n-2i+1$, it
  follows that $x+(n+1)\not\equiv x+n-2i \mod 2$, so
  $g_1^{(n+1)}(\bar{x}_{n+1})\not=g_1^{(n+1)}(\bar{x}_{n-2i})$.\\ Since
  $x+(n+1)+x+1+n-(2i+1)=2x+2n+2-2i-1$, it follows that $x+(n+1)\not\equiv x+1+n-(2i+1)
  \mod 2$ and then
  $g_1^{(n+1)}(\bar{x}_{n+1})\not=g_1^{(n+1)}(\overline{x+1}_{n-(2i+1)})$.\\ From this we
  have: $g_1^{(n+1)}$ is a \gf of $R_{n+1}$, and it is a bipartite digraph. 
  
  Now we will prove that $g_2^{(n+1)}$ is a \gf of $R_{n+1}$. Since there are no arcs
  in $R_{n+1}$ that ends in $\bar{x}_{n+1}$, with $x=0,1,2,3$, and the set
  $\{\bar{0}_{n+1}, \bar{1}_{n+1}, \bar{2}_{n+1}, \bar{3}_{n+1}\}$ is an independent set
  of $R_{n+1}$, we just have to prove that for every $j\in \{0,1,\ldots,n\}$ there exists
  $y\in \Gamma^+_{R_{n+1}}(\bar{x}_{n+1})$  such that $g_2^{(n+1)}(y)=j$.
  But, by definition of $R_{n+1}$, there is an arc from $\bar{x}_{n+1}$ to $\bar{x}_{n},
  \overline{x+1}_{n-1},\bar{x}_{n-2},\overline{x+1}_{n-3},....$.   So, $g_2^{(n+1)}$ is
  a \gf of $R_{n+1}$ and $\max\{g_2^{(n+1)}(u)|u\in V(R_{n+1})\}=n+1$.

  Then with this digraphs, we have prove the following theorem:
  \begin{theorem}\label{R_n}
    For any natural number $n$, there is a digraph $D$ with two \gfs $f$ and $g$ such that
    $\max\{f(x)|x\in V(D)\}-\max\{g(x)|x\in V(D)\}=n$.
  \end{theorem}
  Note that Theorem \ref{R_n} tells us that is important to have bounds for the \gf or
  \sgf.

\end{document}